\documentclass[12pt]{amsart} 
\usepackage{tikz-cd}
\usepackage{fancyhdr} 
\usepackage{latexsym, amssymb, amscd, amsfonts,pstricks,enumitem,amsmath, young}
\usepackage{latexsym, amssymb, amscd, amsfonts,pstricks,enumitem,amsmath}
\usepackage{amsmath,amsthm,amssymb,mathrsfs}
 
\usepackage[all]{xypic}
\usepackage{ifthen}
\usepackage{longtable}
 
\numberwithin{equation}{section}

\renewcommand{\geq}{\geqslant}
\renewcommand{\leq}{\leqslant}
\newcommand{\Osh}{{\mathcal O}}                        

\newcommand{\K}{\mathrm{K}}                            

\newcommand{\Br}{\operatorname{Br}}

\newcommand{\CC}{\mathbb{C}} 
\newcommand{\PP}{\mathbb{P}} 
\newcommand{\QQ}{\mathbb{Q}} 
\newcommand{\ZZ}{\mathbb{Z}} 

\newcommand{\Div}{\mathrm{Div}}

\newtheorem{theorem}{Theorem}[section]

\newtheorem{proposition}[theorem]{Proposition}

\theoremstyle{definition}
\newtheorem{defn}[theorem]{Definition}
\newtheorem{remark}[theorem]{Remark}
\newtheorem{problem}[theorem]{Problem}
\newtheorem{example}[theorem]{Example}

\begin{document}
\title[Contraction for orbifold pairs in dimension $2$]{Numerical contraction for orbifold surfaces}

\author{Nathan Grieve}
\address{Department of Mathematics \& Computer Science,
Royal Military College of Canada, P.O. Box 17000,
Station Forces, Kingston, ON, K7K 7B4, Canada
} 
\address{
School of Mathematics and Statistics, 4302 Herzberg Laboratories, Carleton University, 1125 Colonel By Drive, Ottawa, ON, K1S 5B6, Canada
}
\address{D\'{e}partement de math\'{e}matiques, Universit\'{e} du Qu\'{e}bec \'a Montr\'{e}al, Local PK-5151, 201 Avenue du Pr\'{e}sident-Kennedy, Montr\'{e}al, QC, H2X 3Y7, Canada}

\email{nathan.m.grieve@gmail.com}%

\begin{abstract} 
We study singularities and Artin's contraction theorem for orbifold surfaces.  Our main result has a consequence which is in the direction of the birational Minimal Model Program for b-terminal orbifold surfaces.  For example, we ascertain the nature of extremal contractions for such $b$-terminal pairs.
\end{abstract}
\thanks{Mathematics Subject Classification (2010): 14E30, 14J10. 
}
\thanks{
I hold grants DGECR-2021-00218 and RGPIN-2021-03821 from NSERC (The Natural Sciences and Engineering Research Council of Canada).
}
\thanks{
Date: \today. 
}

\maketitle

\section{Introduction}
A fundamental aspect to the birational geometry of a given $\QQ$-factorial complex projective surface $S$, is the contraction of exceptional Cartier divisors of the first kind.  When $S$ is nonsingular, this is achieved via Castelnuovo's theorem 
\cite[Theorem 5.7]{Hart}.  

More generally, for a given $\QQ$-factorial logarithmic surface $(S,\Delta)$, 
it is of interest to understand the extent to which analogs of Castelnuovo's theorem remain true for those irreducible, reduced and effective Cartier divisors 
$$E \subseteq S$$ 
which have the two properties that: 
\begin{itemize}
\item[(i)]{
$E^2 < 0$; and
}
\item[(ii)]{
$E \cdot (\K_S + \Delta) < 0$.
}
\end{itemize}

Here, we treat the case that $\Delta$ is an \emph{orbifold divisor}
$$
\Delta = \sum_i (1 - 1/m_i) C_i \text{, }
$$
for $m_i$ positive integers, and $C_i$ distinct irreducible curves in a given Gorenstein $\QQ$-factorial complex projective surface $S$.  We say that the pair $(S,\Delta)$ is an \emph{orbifold pair}; we also say that $(S,\Delta)$ is an \emph{orbifold surface} (see also Definition \ref{Orbifold:Surface:Defn}; compare, for example, with \cite[Section 1.2.1]{Campana:2004} or \cite[Section 4]{Ghigi:Kollar:2007}).  

As some motivation, we mention the work \cite{Chan:Ingalls:2005}, which studies log pairs $(S,\Delta)$ with $S$ a nonsingular complex projective surface and $\Delta$ the \emph{ramification divisor} of a given Brauer class $\alpha \in \operatorname{Br}(\CC(S))$.  
For example, in \cite[Theorem 3.10]{Chan:Ingalls:2005}, it is shown that if the \emph{ramification data} of the Brauer class $\alpha$ is \emph{terminal} on $S$, then those irreducible $\K_S + \Delta$-negative curves 
$$E \subseteq S$$ 
with 
$$E^2 < 0$$ 
are in fact minus one curves.  

A key point to the proof of that result, is the work of Artin and Mumford \cite{Artin:Mumford:1971} (see also \cite[Theorem 3.1]{Chan:Ingalls:2005}).  Recall, that this circle of ideas has origins in the work of Clemens and Griffiths \cite{Clemens:Griffiths:1972}, in particular, the problem of constructing unirational varieties which are not rational, and the algebraic conic bundle approach which was developed in \cite{Artin:Mumford:1971}.

As some further motivation for the study of orbifold logarithmic pairs, especially from the point of view of fibrations and orbifold stacks, we refer to \cite{Campana:2004}, \cite{Matsuki:Olsson:2005}, \cite{Ghigi:Kollar:2007}, \cite{Abramovich:Hassett:2011} and \cite{Ross:Thomas:2011}.  In particular, our results here (Theorems \ref{numerical:contraction:thm} and \ref{MMP:claim} below) may be of interest in terms of developing the birational perspective of these works in further detail.  
  
Our first result, see Theorem \ref{numerical:contraction:thm} below, gives \emph{numerical criteria}  for contraction of curves on \emph{orbifold surfaces}.  Before stating it, we recall, in precise terms, the concept of \emph{orbifold surface}, which is relevant for our purpose here.

\begin{defn}\label{Orbifold:Surface:Defn}
By an \emph{orbifold surface}, is meant an \emph{orbifold pair} $(S,\Delta)$ where $S$ is a $\QQ$-factorial surface and where $\Delta$ is an \emph{orbifold divisor}
$$
\Delta = \sum_i (1 - 1/m_i) C_i \text{, }
$$
for $m_i$ positive integers, and $C_i$ distinct irreducible curves in $S$.  Note, in particular, that 
 the \emph{orbifold canonical divisor} $\K_S + \Delta$ is $\QQ$-Cartier.
\end{defn}   

For our purposes, we say that a $\QQ$-factorial surface $S$ is \emph{Gorenstein}, if its canonical divisor $\K_S$ is a Cartier divisor.  Theorem \ref{numerical:contraction:thm} is then formulated in the following way.

\begin{theorem}\label{numerical:contraction:thm}
Let $S$ be a Gorenstein $\QQ$-factorial projective complex surface and $(S,\Delta)$ an orbifold pair:
$$
\Delta = \sum_i \left(1 - 1 / m_i \right) C_i,
$$
for $m_i$ positive integers and $C_i$ irreducible curves in $S$.  Suppose that $\Delta$ has simple normal crossings support.  
Let $E$ be a reduced irreducible effective Cartier divisor in $S$, which has the property that $S$ is nonsingular along $E$.  Suppose further that $E$ satisfies the two conditions that:
\begin{enumerate}
\item[(i)]{$E^2 < 0$; and}
\item[(ii)]{$(\K_S + \Delta) \cdot E < 0$.}
\end{enumerate}
Then $E$ is a smooth rational curve and is contracted by a morphism 
$$\pi \colon S \rightarrow S'$$ 
onto a normal projective surface $S'$.  
Further, if $E$ is not a component of $\operatorname{Supp}(\Delta)$, then $E$ is a $(-1)$-curve.  On the other hand, assume that $E$ is a component of $\operatorname{Supp}(\Delta)$ and write:
$$
\Delta = \Delta' + \left(1 - 1 / e \right) E,
$$
with $E$ not a component of $\Delta'$.
Then:
$$
E \cdot \lceil \Delta' \rceil \leq 3.
$$
If equality holds, then:
$$
E \cdot \Delta' = 3 - 1/a - 1/b - 1/c
$$
for $(a,b,c)$ one of the Platonic triples:
\begin{itemize}
\item{$(2,3,5)$;}
\item{$(2,3,4)$;}
\item{$(2,3,3)$; and}
\item{$(2,2,a)$,}
\end{itemize}
for $a\geq 2$.  Finally, in case that 
$$ E \cdot \lceil \Delta' \rceil < 3\text{,}$$ 
then $E$ is a $(-1)$-curve provided that:
$$
E \cdot \Delta' = (1 - 1/(em)) + (1 - 1/(en)),
$$
for some positive integers $m$ and $n$.
\end{theorem}

We prove Theorem \ref{numerical:contraction:thm} in Section \ref{proof:orbifold:numerical:contraction}.  
It is deduced from the contraction theorem of Artin \cite{Artin:1962}.

\begin{theorem}[Artin's criteria for contraction]\label{Artin:Contraction}
Let $C$ be a connected curve in a normal projective surface $S$ and let $C_i$ be the irreducible components of $C$.  Assume that $S$ is nonsingular along $C$.  Then $C$ is contractible via a morphism 
$$\pi \colon S \rightarrow S'$$ 
to a normal projective surface $S'$ if and only if:
\begin{enumerate}
\item[(i)]{the intersection matrix $||(C_i \cdot C_j) ||$ is negative definite; and
}
\item[(ii)]{
each effective cycle $Z > 0$ with support on $C$ has nonpositive arithmetic genus.
}
\end{enumerate}
\end{theorem}
\begin{proof}
This is \cite[Theorem 2.3]{Artin:1962}. (Compare also with \cite[Theorem 1.2]{sakai:1984}.) 
\end{proof}

By using the uniqueness part of the Contraction Theorem for klt pairs, \cite[Theorem 3.7]{Kollar:Mori:1998}, Theorem \ref{numerical:contraction:thm} has implications in the direction of the b-minimal model program for $(S,\mathbb{D})$, a \emph{simple normal crossing b-terminal pair}.  This result can be compared with \cite[Theorem 3.19]{Chan:Ingalls:2005} which obtains a slightly more refined statement that applies for those orbifold surfaces which are determined by Brauer classes.  We prove Theorem \ref{MMP:claim} in Section \ref{b-orbifold-dim-2}.  

\begin{theorem}\label{MMP:claim}
Suppose that $(S,\mathbb{D})$ is a b-terminal orbifold pair with $\mathbb{D}_S$ a simple normal crossing divisor and $S$ a nonsingular complex projective surface.  Then either $\K_S + \mathbb{D}_S$ is nef or there exists an extremal curve $E$ such that 
$$E \cdot (\K_S + \mathbb{D}_S) < 0$$ and exactly one of the following occurs.  
\begin{itemize}
\item{If 
$E^2 < 0$, then $E$ is a negative rational curve and is contracted by a morphism 
$$\pi \colon S \rightarrow S'\text{,}$$ 
to a $\QQ$-factorial surface $S'$.  If $S$ is nonsingular along $E$, then $\pi$ is the contraction given by Theorem \ref{numerical:contraction:thm}.  
}
\item{If 
$E^2 = 0$, then 
$$\pi \colon S \rightarrow C$$ 
is a ruled surface with $E$ a fibre and $-(\K_S + \mathbb{D}_S)$ is relatively ample for the map $\pi$.
}
\item{If 
$E^2 > 0$, then $S \simeq \PP^2$ and $-\K_S + \mathbb{D}_S)$ is ample.  
}
\end{itemize}
\end{theorem}

As some additional motivation for the results of the present article, we note that the conclusion of Theorem \ref{numerical:contraction:thm}, which is formulated within the context of orbifold surfaces, is weaker than that of \cite[Theorem 3.10]{Chan:Ingalls:2005}.  In particular, it prevents a direct establishment of an orbifold minimal model program along the lines of what is done in \cite[Section 3.5]{Chan:Ingalls:2005}.  

But on the other hand, in Section \ref{birationial:discrepancies}, we discuss the concept of discrepancy for birational orbifold pairs.  Then, in Proposition \ref{orbifold:contraction:proposition}, we show how, the condition of positive birational discrepancy (which we call $b$-terminal) for orbifold surfaces, combined with Theorem \ref{numerical:contraction:thm}, allows for construction of a contraction morphism by analogy with \cite[Proposition 3.14]{Chan:Ingalls:2005}.

We believe that this slight subtlety, which is observed here, amongst  those orbifold surfaces which arise from ramification divisors of Brauer classes, compared to general orbifold surfaces, is of some independent interest.  It provides some motivation for future investigation, in addition to what we do here.  

Moreover, our investigation, in the present article, raises the question as to how the ramification divisors of Brauer classes, as in the work of Chan and Ingalls \cite{Chan:Ingalls:2005}, fit into the broader context of orbifold divisors.  We formulate, in precise terms, one approach to this problem as Problem \ref{ramification:orbifold:characterization:dim:2} below.

\begin{problem}\label{ramification:orbifold:characterization:dim:2}
Let $S$ be a nonsingular complex projective surface, with function field $\mathbb{C}(S)$, and consider an orbifold pair $(S,\Delta)$, with the orbifold divisor $\Delta$ having simple normal crossings support.  Give necessary and sufficient conditions for the boundary orbifold divisor $\Delta$ to have shape
$$
\Delta = \mathbb{D}(\alpha)_S = \Delta_\alpha \text{,}
$$
for $\mathbb{D}(\alpha)_S$ the trace on $S$ of the ramification birational divisor $\mathbb{D}(\alpha)$, which is determined by $\alpha$ an element of $\Br(\CC(S))$, the Brauer group of $\CC(S)$.
\end{problem}

As another direction for future investigation, our discussion of birational discrepancy for birational orbifold pairs (Section \ref{birationial:discrepancies}), suggests the possibility for its interpretation within the context of generalized pairs, in the sense of Birkar and Zhang \cite[Definition 1.4 and Section 4]{Birkar:Zhang:2016}.  We do not pursue a further development of that view point here.

\section{Preliminaries}

By a \emph{surface}, unless explicitly stated otherwise, we mean a \emph{Gorenstein}, \emph{$\QQ$-factorial}, \emph{irreducible}, \emph{projective}, \emph{complex surface} $S$.  A curve $C$ in $S$ is \emph{contractible} if there exists a map 
$$\pi \colon S \rightarrow S'$$ 
onto a surface $S'$ which is a single point.  Such a map $\pi$ is uniquely determined and is called the \emph{contraction} of $C$.  By a \emph{simple normal crossing} divisor on $S$, we mean a Carter divisor $\Delta$ on $S$, with rational coefficients, and having the property that each of the irreducible components of $\operatorname{Supp}(\Delta)$ are nonsingular and each pair of these irreducible components intersect everywhere transversally.

We recall, similar to \cite{sakai:1984}, a description of intersection theory for curves in $S$.  The intersection pairing:
$$
\Div_\QQ(S) \times \Div_\QQ(S) \rightarrow \QQ.
$$
may be described in the following way.  

We choose a resolution of singularities 
$$\pi \colon S' \rightarrow S$$ 
with exceptional set: 
$$\operatorname{Exc}(\pi) = \bigcup_i E_i.$$  
For a (Weil) divisor 
$$D \in \Div_\QQ(S) \text{,}$$ 
the \emph{inverse image} in $S'$ is given by:
$$
\pi^* D := \pi_*^{-1}D + \sum a_i E_i.
$$
Here, $\pi_*^{-1}D$ is the strict transform of $D$ with respect to $\pi$; the rational numbers $a_i$ are uniquely determined by the equations:
$$
\pi_*^{-1}D \cdot E_j + \sum_i a_i E_i \cdot E_j = 0.
$$
for all $j$.

Given two divisors $D$ and $D'$ in $\Div_\QQ(S)$, their \emph{intersection number} is then defined to be the (well defined) quantity:
$$
D \cdot D' := (\pi^* D) \cdot (\pi^* D').
$$

Let $E$ be an effective Cartier divisor in $S$ and assume that $S$ is nonsingular along $E$.  Then, as in \cite[p. 486]{Artin:1962}, $E$ has \emph{arithmetic genus} given by:
\begin{equation}\label{arithmetic:genus:formula}
p_a(E) = \frac{1}{2} \left( (E^2) + (E \cdot \K_S) \right) + 1.
\end{equation}
In particular:
$$
1 - p_a(E) = \chi(\Osh_E) = h^0(E,\Osh_E) - h^1(E,\Osh_E);
$$
{see \cite[Equation (1.1)]{Artin:1962}.}

\section{A numerical criterion for contraction}\label{proof:orbifold:numerical:contraction}

Here, we establish Theorem \ref{numerical:contraction:thm}.

\begin{proof}[Proof of Theorem \ref{numerical:contraction:thm}]
The proof is similar to the proof of \cite[Theorem 3.10]{Chan:Ingalls:2005}.
We distinguish amongst two cases:
\begin{itemize}
\item{
{\bf Case 1.} The divisor $E$ is not contained in the discriminant curve $\operatorname{Supp}(\Delta)$.
}
\item{
{\bf Case 2.}  The divisor $E$ is contained in the discriminant curve $\operatorname{Supp}(\Delta)$. In particular, $E$ is a component of $\Delta$.
}
\end{itemize}

In {\bf Case 1}, the intersection number $E . \Delta$ is non-negative:
$$ E \cdot \Delta \geq 0;$$
it then follows that the intersection number $E.\K_S$ is negative:
$$ E \cdot \K_S \leq -1 \text{.}$$
On the other hand, since $E$ has negative self-intersection number:
$$ E^2 \leq -1,$$
the arithmetic genus formula, \eqref{arithmetic:genus:formula}, 
implies, since  
$$E \cdot (E + \K_S) \leq -2,$$
that $E$ has arithmetic genus equal to zero:
$$ p_a(E) = 0.$$
In particular, $E$ is a smooth rational curve.

Finally, note (since both $E^2$ and $E \cdot \K_S$ are negative) that the genus formula also implies that:
$$ E \cdot \K_S = E^2 = -1;$$
thus $E$ is a $(-1)$-curve.  In particular, there exists a morphism 
$$\pi \colon S \rightarrow S'$$ 
which contracts exactly $E$.

Next, note that
$$ E \cdot \Delta < 1 $$
since
$$ (\K_S + \Delta) \cdot E < 0$$
and
$$ \K_S \cdot E = -1.$$
For the sake of completeness, note that the coefficients of $\Delta$ are of the form:
$$
1 - 1 / e_i \geq 1 / 2.
$$
One consequence is that $E$ can meet $\operatorname{Supp}(\Delta)$ in at most one point.  In particular, we contract $E$ via 
$$\pi \colon S \rightarrow S'\text{.}$$  
This completes the proof of {\bf Case 1}.

We now consider {\bf Case 2}.
By assumption, $E$ is a component of $\Delta$.  Write:
$$
\Delta = \Delta' + (1-1/e)E.
$$
Then:
$$
E \cdot \Delta' - (1/e)E^2 > 0
$$
and
$$
0 > E \cdot (\K_S + \Delta) = E \cdot \K_S + E^2 + \left(E \cdot \Delta' - \frac{1}{e} E^2 \right)\text{,}
$$
so:
$$
E \cdot (\K_S + E)<0;
$$
the genus formula \eqref{arithmetic:genus:formula} then implies that:
$$
p_a(E) = 0.
$$
Again, $E$ is a smooth rational curve and
we have satisfied the hypothesis of the contraction theorem of Artin.

By assumption, the discriminant curve $\operatorname{Supp}(\Delta)$ has simple normal crossings support.  Thus:
$$
E \cdot \Delta' = \sum \left( 1 - \frac{1}{m_i} \right),
$$
for suitable integers $m_i \geq 2$.  Also: 
$$E \cdot \Delta' = \sum_i \left( 1 - 1 / m_i \right) < 2\text{,}$$
and
$$ m_i \in \ZZ_{\geq 2};$$
it follows that:
$$E \cdot \lceil \Delta' \rceil \leq 3.$$   

Next suppose that
$$E \cdot \lceil \Delta' \rceil = 3.$$  
The only possible triples 
$$(m_1,m_2,m_3)\text{,} $$
with $m_i \in \ZZ_{\geq 1}$, which have the property that 
$$2 - E \cdot \Delta' > 0$$ 
are the Platonic triples:
\begin{itemize}
\item{$(2,3,5)$;}
\item{$(2,3,4)$;}
\item{$(2,3,3)$;  and} 
\item{$(2,2,a)$, }
\end{itemize}
for $a\geq 2$.  

Suppose now that:
$$E \cdot \lceil \Delta' \rceil = 2$$
and
$$
E \cdot \Delta' = 2 - 1/(em) -1/(en),
$$ 
for positive integers $m$ and $n$.  In this context, since:
$$
1 \leq - E^2 < e(2 - E \cdot \Delta') = 1/m + 1/n,
$$
either $m$ or $n$ equals one.  Thus:
$$
1 \leq - E^2 < 2
$$
and so $E$ is a $(-1)$-curve as desired.
\end{proof}

\begin{remark}
Note that in {\bf Case 2} of the above proof, we also have the lower bound:
$$E^2 > e(-2 + E \cdot \Delta').$$
Further, in both {\bf Case 1} and {\bf Case 2}, it also holds true that:
$$
-2 = E \cdot (\K_S + E).
$$
\end{remark}

\begin{remark}
It would be interesting if an argument along the lines of Theorem \ref{numerical:contraction:thm} could be used to characterize $(-1)$-curves which have the properties that 
$$E^2 < 0$$ 
and 
$$(\K_S + \Delta) \cdot E < 0$$ 
for $\Delta$ an orbifold divisor.  For example, if $E$ is a component of $\Delta$, could this condition imply that $E$ is a $(-1)$-curve if and only if $\Delta_\alpha$ is the ramification divisor of some Brauer class $\alpha \in \operatorname{Br}(\mathbb{C}(S))$ having terminal ramification data?
\end{remark}

\section{Pairs and discrepancies}\label{pairs:descrep}

In this section, we make precise our conventions about pairs and discrepancies.  We work over the complex numbers $\CC$.  Our conventions are consistent with \cite{Kollar:Mori:1998}. 

 Let $(X,\Delta)$ be a \emph{pair} with $X$ a proper normal variety and 
$$\Delta = \sum a_i D_i$$ 
a $\QQ$-divisor.  Suppose that $(X,\Delta)$ is \emph{$\QQ$-Gorenstein}, i.e., that $\K_X + \Delta$ is $\QQ$-Cartier.  

The \emph{discrepancies} $a(E;X,\Delta)$ of exceptional prime divisors $E$ over $X$, with respect to $\Delta$ are defined by the conditions that:
$$ \K_Y + f_*^{-1} \Delta \equiv_{\mathrm{num}} f^*(\K_X + \Delta) + \sum_{\substack{E \text{ is an exceptional prime divisor} \\ \text{ over $X$} }} a(E;X,\Delta) E.$$
Here 
$$f \colon Y \rightarrow X$$ 
is a proper birational morphism from a normal variety $Y$ and $f_*^{-1} \Delta$ denotes the \emph{strict} (or \emph{birational}) transform of $\Delta$.  

The \emph{discrepancy} of the pair $(X,\Delta)$ is then defined to be:
$$ \operatorname{discrep}(X,\Delta) = \inf_{\substack{E \text{ is an exceptional prime divisor} \\ \text{ over $X$} }} \{ a(E;X,\Delta)  \}.$$

In line with \cite{Kollar:Mori:1998}, we say that $(X,\Delta)$ is \emph{terminal} if: 
$$\operatorname{discrep}(X,\Delta) > 0,$$ 
\emph{canonical} if: 
$$\operatorname{discrep}(X,\Delta) \geq 0,$$ 
\emph{log canonical} (\emph{lc}) if:
$$
\operatorname{discrep}(X,\Delta) \geq -1,
$$
and \emph{Kawamata log terminal} (\emph{klt}) if: 
$$\operatorname{discrep}(X,\Delta) > -1 $$ 
and 
$$\lfloor \Delta \rfloor \leq 0 \text{.}$$

More generally, as in \cite{Shokurov:1996}, fixing a positive real number $\epsilon > 0$, we say that  
$(X,\Delta)$ is \emph{$\epsilon$-terminal} if: 
$$\operatorname{discrep}(X,\Delta) > \epsilon,$$ 
\emph{$\epsilon$-canonical} if: 
$$\operatorname{discrep}(X,\Delta) \geq \epsilon,$$ 
\emph{$\epsilon$-plt} if: 
$$
\operatorname{discrep}(X,\Delta) > \epsilon -1,
$$
\emph{$\epsilon$-lc} if:
$$
\operatorname{discrep}(X,\Delta) \geq \epsilon -1,
$$
and \emph{$\epsilon$-klt} if: 
$$\operatorname{discrep}(X,\Delta) > \epsilon -1 $$ 
and 
$$\lfloor \Delta \rfloor \leq 0\text{.}$$

\section{$b$-log pairs and $b$-discrepancies}\label{birationial:discrepancies}

We fix our conventions about $b$-log pairs and $b$-discrepancies for $b$-divisors (compare with \cite{Shokurov:1996} and \cite{Chan:Ingalls:2005}).  

Let $(X,\mathbb{D})$ be a proper \emph{fractional $b$-log pair}.  Then $X$ is a proper normal variety, $\mathbb{D}$ is a $b$-divisor on $X$, the divisor $\K_X + \mathbb{D}_X$ is $\QQ$-Cartier and the coefficients of $\mathbb{D}_Y$, for all proper normal models $Y$ of $X$, lie in $[0,1) \cap \QQ$.  In particular, $(X,\mathbb{D}_X)$ is a pair. Here $\mathbb{D}_Y$ denotes the \emph{trace} of $\mathbb{D}$ on $Y$ a proper normal model of $X$.  

For each prime divisor $F$ over $X$, let 
$$d_F \in [0,1) \cap \QQ$$ 
be the coefficient of $\mathbb{D}$ along $F$.  Define the \emph{ramification index} 
$$
r_F \in [1,\infty) \cap \QQ
$$ 
of $\mathbb{D}$ along $F$ by:
$$ 
r_F = \frac{1}{1-d_F} \in [1, \infty) \cap \QQ .
$$
Equivalently:
$$ 
d_F = 1 - \frac{1}{r_F} \in [0,1) \cap \QQ. 
$$

  The \emph{$b$-discrepancies}, of exceptional divisors $E$ over $X$, with respect to $(X,\mathbb{D})$ are then defined by:
$$ 
b(E;X,\mathbb{D}) + 1 = r_E ( a(E;X,\mathbb{D}_X) + 1 ).
$$

  Put:
\begin{equation}\label{b:discrep:relation}
b'(E;X,\mathbb{D}) = \frac{b(E;X,\mathbb{D})}{r_E}.
\end{equation}
Then note:
$$ 
\K_Y + \mathbb{D}_Y \equiv f^*(\K_X + \mathbb{D}_X) + \sum_{\substack{E \text{ is a prime $f$-exceptional } \\ \text{ on $Y$} }} b'(E;X,\mathbb{D}) E
$$
for all proper normal models 
$$f \colon Y \rightarrow X$$ 
of $X$. Indeed, this follows since:
\begin{equation}\label{a:b:discrep:relation}
b'(E'X,\mathbb{D}) = a(E; X, \mathbb{D}_X) + 1 - \frac{1}{r_E}
\end{equation}
combined with the two relations:
$$
\K_Y + \mathbb{D}_Y = \K_Y + f_*^{-1} \mathbb{D}_X + \sum_E \left(1 - \frac{1}{r_E} \right) E 
$$ 
and
\begin{multline*}
\K_Y + f_*^{-1} \mathbb{D}_X + \sum_E \left(1 - \frac{1}{r_E} \right) E 
= \\
 f^*(\K_X + \mathbb{D}_X) + \sum_E \left(a(E;X,\mathbb{D}_X) + 1 - \frac{1}{r_E} \right) E.
 \end{multline*}
These relations also highlight the intuition behind the concept of b-discrepancy.

Set:
$$ 
\operatorname{b-discrep}(X,\mathbb{D}) := \inf_{\substack{\text{exceptional prime} \\
\text{ divisors $E$ over $X$}}}  \{b(E;X,\mathbb{D}) \}.
$$
We say that $(X,\mathbb{D})$ is \emph{$b$-terminal} if: 
$$
\operatorname{b-discrep}(X,\mathbb{D}) > 0,
$$
that $(X,\mathbb{D})$ is \emph{$b$-canonical} if: 
$$
\operatorname{b-discrep}(X,\mathbb{D}) \geq 0,
$$
that $(X,\mathbb{D})$ is \emph{$b$-log terminal} ($b$-lt) if:
$$
\operatorname{b-discrep}(X,\mathbb{D}) > - 1,
$$
and that $(X,\mathbb{D})$ is \emph{$b$-log canonical} ($b$-lc) if:
$$
\operatorname{b-discrep}(X,\mathbb{D}) \geq - 1.
$$
Note that these concepts are also formulated in 
\cite{Shokurov:1996}.  (See also \cite[Section 3.5]{Grieve:Ingalls:2016}.)

  Fix $\delta \geq -1$ and $\epsilon > 0$.  We then want to compare the discrepancies $a(E; X, \mathbb{D}_X)$, $b(E;X,\mathbb{D})$ and $b'(E;X,\mathbb{D})$ for $E$ a fixed prime exceptional divisor over $X$.  Motivated by our above concepts of $\epsilon$-discrepancy, here we make note of the following elementary remark.

\begin{proposition}\label{discrepancy:prop}
Fix $\delta \geq -1$ and $\epsilon > 0$.
The following assertions hold true.
\begin{enumerate}
\item[(i)]{
If 
$$b(E;X,\mathbb{D}) >  \epsilon + \delta,$$ then 
$$b'(E;X,\mathbb{D}) >  \frac{\epsilon + \delta}{r_E}$$
and
$$
a(E;X,\mathbb{D}_X) > \frac{\epsilon + \delta + 1}{r_E} - 1.
$$
}
\item[(ii)]{
If
$$
a(E;X,\mathbb{D}_X) > \epsilon + \delta,  
$$
then
$$
b(E;X,\mathbb{D}) > r_E(\epsilon + \delta + 1) - 1
$$
and
$$
b'(E;X,\mathbb{D}) > \epsilon + \delta + 1 - \frac{1}{r_E}.
$$
}
\end{enumerate}
The same inequalities holds true by replacing $>$ with $\geq$.
\end{proposition}
\begin{proof}
Each of the inequalities are easily deduced from the relations given by \eqref{b:discrep:relation} and \eqref{a:b:discrep:relation}.
\end{proof}

Motivated by those $b$-divisors which arise from Brauer classes, for example \cite{Chan:Ingalls:2005}, combined with the concept of orbifold pairs in the sense of Campana, for instance \cite{Campana:2004}, we define a general concept of \emph{orbifold $b$-divisor}.

\begin{defn}
Let $(X,\mathbb{D})$ be a proper \emph{fractional $b$-log pair}.  We say that $(X,\mathbb{D})$ is a \emph{$b$-orbifold pair} if the ramification indices associated to each trace of $\mathbb{D}$ are integers.
\end{defn}

We also record the following remark, by analogy with \cite[Proposition 3.15]{Chan:Ingalls:2005}. 

\begin{proposition}\label{b-terminal-klt}   Let $(X,\mathbb{D})$ be a b-orbifold pair.  If  $(X,\mathbb{D})$ is b-terminal, then the pair $(X,\mathbb{D}_X)$ is klt.   Conversely, if the pair $(X,\mathbb{D}_X)$ is klt, then the pair $(X,\mathbb{D})$ is b-log terminal. 
\end{proposition}

\begin{proof}
We argue as in \cite[Proposition 3.15]{Chan:Ingalls:2005}.  Specifically, as in \cite[Corollary 2.32]{Kollar:Mori:1998}, there is a log resolution of $(X,\mathbb{D}_X)$ which we may use to compute discrepancy.  Fix such a resolution 
$$f \colon X' \rightarrow X.$$  
We then study the discrepancy equations:
$$
\K_{X'} + \mathbb{D}_{X'} \equiv f^*(\K_X + \mathbb{D}_X) + \sum b'(X,\mathbb{D};E_i) E_i
$$
and
$$
\K_{X'} + f^{-1}_* \mathbb{D}_X \equiv \pi^*(\K_X + \mathbb{D}_X) + \sum a(X,\mathbb{D}_X; E_i) E_i.
$$
In particular:
$$
\operatorname{b-discrep}(X,\mathbb{D}) = \inf_{E_i} \{ r_{E_i} b'(X,\mathbb{D};E_i) \}
$$
and
$$
\operatorname{discrep}(X,\mathbb{D}_X) = \inf_{E_i} \{a(X,\mathbb{D}_X;E_i) \}.
$$
Applying Proposition \ref{discrepancy:prop}, the conclusion desired by Proposition \ref{b-terminal-klt} then follows.
\end{proof}

\begin{example}
In order to illustrate the above concepts, here we indicate the manner in which the concept of \emph{terminal models} from \cite[Corollary 1.4.3]{Birkar:Cascini:Hacon:McKernan} relates to the concept of $b$-terminal models for a given projective fractional $b$-log pair $(X,\mathbb{D})$.  Fix a projective log resolution 
$$f \colon X'\rightarrow X\text{,}$$ so that the pair $(X',\mathbb{D}_{X'})$ is klt.  Let 
$$g \colon X'' \rightarrow X'$$ 
be a \emph{terminal model}, in the sense of \cite[p. 413]{Birkar:Cascini:Hacon:McKernan}, of the klt pair $(X',\mathbb{D}_{X'})$.    Then, in particular, $X''$ is $\QQ$-factorial, the pair $(X'',\mathbb{D}_{X''})$ is terminal and it follows that the fractional $b$-log pair $(X'',\mathbb{D})$ is $b$-terminal.
\end{example}

\section{$b$-Orbifold divisors in dimension $2$}\label{b-orbifold-dim-2}

  By analogy with \cite[Proposition 3.13]{Chan:Ingalls:2005}, we now study the condition that:
$$ 
\operatorname{b-discrep}(S,\mathbb{D}) > 0.
$$ 

\begin{proposition}\label{orbifold:contraction:proposition}
Let $(S,\mathbb{D})$ be an orbifold surface with 
$$\operatorname{b-discrep}(S,\mathbb{D}) > 0\text{.}$$  
Then $S$ admits a nonsingular model 
$$\pi \colon S' \rightarrow S$$ 
that has the following two properties.
\begin{enumerate}
\item[(i)]{The pair $(S',\mathbb{D}_{S'})$ has simple normal crossings support.}
\item[(ii)]{If $S' \not = S$, then $S'$ admits an irreducible curve $E$ with the property that:
$$
(\K_{S'} + \mathbb{D}_{S'}) \cdot E < 0
$$
and
$$
E^2 < 0.
$$
}
\end{enumerate}
\end{proposition}
\begin{proof}
We argue similar to \cite[Proposition 3.14]{Chan:Ingalls:2005}.  
Suppose that:
$$\operatorname{b-discrep}(S,\mathbb{D}) > 0.$$
Let 
$$\pi \colon S' \rightarrow S$$ 
be a simple normal crossing resolution of singularities for $(S,\mathbb{D}_S)$.  Consider the equation:
$$
\K_{S'} + \mathbb{D}_{S'} \equiv \pi^*(\K_S + \mathbb{D}_S) + \sum a_i E_i
$$
and take the intersection product with the divisor $\sum a_i E_i$.  The intersection matrix 
$$||(E_i \cdot E_j)||$$ 
is negative definite.  Furthermore
$$
E_i \cdot \pi^* C = 0,
$$
for all irreducible curves $C$ on $S$;
it then follows that:
$$
(\K_{S'} + \mathbb{D}_{S'}) \cdot \sum a_i E_i = \left(\sum a_i E_i \right)^2 < 0.
$$
So, there exists 
$$E = E_j$$ 
with the property that:
$$
\left( \K_{S'} + \mathbb{D}_{S'} \right) \cdot E < 0
$$
and 
$$
E^2 < 0.
$$
We can then apply the contraction theorem (Theorem \ref{numerical:contraction:thm}).  In particular, we contract $E_j$ to get something smaller than $(S',\mathbb{D}_{S'})$. 
\end{proof}

 We now use the Cone Theorem for klt pairs, \cite[Theorem 3.5]{Kollar:Mori:1998}, to prove Theorem \ref{MMP:claim}.

\begin{proof}[Proof of Theorem \ref{MMP:claim}]
We argue as in \cite[Theorem 3.19]{Chan:Ingalls:2005}, and, in particular, consider consequence of \cite[Theorem 1.28]{Kollar:Mori:1998}. 

If $\K_S + \mathbb{D}_S$ is not nef, then existence of such an extremal curve $E$ is implied by the rationality theorem for klt pairs, \cite[Theorem 3.5]{Kollar:Mori:1998}.  

In case that 
$$E^2 < 0\text{,}$$ 
then we apply Theorem \ref{numerical:contraction:thm} combined with the uniqueness of divisorial contractions of extremal rays \cite[Theorem 3.7]{Kollar:Mori:1998}.  

In the other cases, we are able to show that 
$$\K_S \cdot E < 0$$ 
and so, as explained in the proof of \cite[Theorem 3.19]{Chan:Ingalls:2005}, we are in the situation of \cite[Theorem 1.28]{Kollar:Mori:1998}.
\end{proof}

\subsection*{Acknowledgements}  
This work benefitted, respectively, from visits to CIRGET (Montreal), NCTS (Taipei), Academia Sinica (Taipei), and BIRS (Banff),  during the Spring, Summer and Fall of 2019. I thank Colin Ingalls, Steven Lu, Carlo Gasbarri and colleagues for discussions and comments on related topics.  Finally, I thank the Natural Sciences and Engineering Research Council of Canada for their support via my grants DGECR-2021-00218 and RGPIN-2021-03821.

\providecommand{\bysame}{\leavevmode\hbox to3em{\hrulefill}\thinspace}
\providecommand{\MR}{\relax\ifhmode\unskip\space\fi MR }
\providecommand{\MRhref}[2]{%
  \href{http://www.ams.org/mathscinet-getitem?mr=#1}{#2}
}
\providecommand{\href}[2]{#2}

\end{document}